\documentclass[a4paper,12pt]{amsart}

\usepackage{amssymb}
\usepackage{amsfonts}
\usepackage{amsmath}
\usepackage{mathrsfs}
\usepackage{amsthm}
\usepackage{enumerate}
\usepackage{hyperref}
\usepackage[utf8]{inputenc}
\usepackage[T1]{fontenc}
\usepackage[british]{babel}
\usepackage{geometry}
\usepackage{fullpage}

\newtheorem{theorem}[equation]{Theorem}

\newtheorem{proposition}[equation]{Proposition}

\newtheorem{definition}[equation]{Definition}
\newtheorem{example}[equation]{Example}

\newtheorem{remark}[equation]{Remark}

\numberwithin{equation}{section}

\newcommand{\rn}{\mathbb{R}}
\newcommand{\nn}{\mathbb{N}}

\newcommand{\cn}{\mathbb{C}}

\newcommand{\scj}{\subseteq}

\newcommand{\T}{\mathbb{T}}

\newcommand{\bx}{\mathbf{x}}

\newcommand{\CB}{\mathcal{B}}

\newcommand{\CG}{\mathcal{G}}

\newcommand{\CI}{\mathcal{I}}

\newcommand{\CO}{\mathcal{O}}
\newcommand{\CP}{\mathcal{P}}

\newcommand{\CR}{\mathcal{R}}
\newcommand{\CT}{\mathcal{T}}

\newcommand{\CCC}{\mathscr{C}}

\newcommand{\falg}{\CB_{\CG}}
\newcommand{\falgx}{\CB_{X}}
\newcommand{\gp}{{(\CG,\CR)}}
\newcommand{\triple}{{(\CG,\CR_N,\CR_S)}}

\newcommand{\Hom}{\operatorname{Hom}}

\newcommand{\newfd}[5]{\begin{array}{cccc}{#1}: &{#2} &\longrightarrow &{#3} \\& {#4} &\longmapsto &{#5}\end{array}}

\DeclareFontFamily{U}{matha}{\hyphenchar\font45}
\DeclareFontShape{U}{matha}{m}{n}{
	<5> <6> <7> <8> <9> <10> gen * matha
	<10.95> matha10 <12> <14.4> <17.28> <20.74> <24.88> matha12
}{}
\DeclareSymbolFont{matha}{U}{matha}{m}{n}
\DeclareFontSubstitution{U}{matha}{m}{n}

\DeclareFontFamily{U}{mathx}{\hyphenchar\font45}
\DeclareFontShape{U}{mathx}{m}{n}{
	<5> <6> <7> <8> <9> <10>
	<10.95> <12> <14.4> <17.28> <20.74> <24.88>
	mathx10
}{}
\DeclareSymbolFont{mathx}{U}{mathx}{m}{n}
\DeclareFontSubstitution{U}{mathx}{m}{n}

\DeclareMathDelimiter{\vvvert}{0}{matha}{"7E}{mathx}{"17}

\title{Notes on universal C*-algebras}
\author{Gilles G. de Castro}

\begin{document}
	
	\maketitle
	
	\begin{abstract}
		In these notes, we explain in details the construction of the universal C*-algebra as done by Blackadar, starting with the construction of the free *-algebra. At the end, we explain the extension done by Boava and the author to include relations described using the strong operator topology. As an example, we show that a graph C*-algebra can be defined using infinite sums.
	\end{abstract}
	
	\tableofcontents
	
	\section{Introduction}
	
	In abstract algebra, it is often useful to present an object in terms of generators and relations. A quintessential example is that of the complex numbers, that one can think of as being polynomials over the real numbers in the variable $i$ satisfying the relation $i^2=-1$. One way of achieving this is by constructing a free object in the corresponding category and then taking a suitable quotient to impose the needed relations. One can interpret the free object as that given by imposing no relations.
	
	When working with C*-algebras, the above framework does not work because, except for the empty set, there is no free C*-algebra generated by a set (Proposition \ref{prop:free.C*-algebra}). One approach to avoid this difficulty is to present operators in a Hilbert space satisfying the relations that we want to impose and then take the C*-algebra generated by these operators. The main issue is that the C*-algebra built this way depends on the representation and a question of interest is whether this C*-algebra is uniquely defined, that is, it does not depend on the representation. Examples of uniqueness theorems can be found, for instance, in the results of Coburn describing C*-algebras generated by an isometry \cite{Coburn}, the simplicity results of Cuntz algebras \cite{Cuntz} and the uniqueness results of Cuntz-Krieger algebras \cite{CuntzKrieger}. Another approach is to find a universal representation so that all other representations factors through the universal one.
	
	A general construction of universal C*-algebras was presented by Blackadar in \cite{Blackadar}. Although he said that ``we can allow any kind of relations which can be formulated for operators in a Hilbert space or for elements in a C*-algebra'' (as long as the relations are admissible), in his paper, he only developed the theory for a specific kind of relations, that we call norm relations here (see Section \ref{s:norm.relation}). Nevertheless, even after Blackadar's paper, there are results proving the existence of a universal representation when only using norm relations. For instance, the existence of a universal representation for Cuntz-Krieger algebras of finite matrices is proved in \cite{UniversalCuntzKrieger}, for row-finite graphs \cite{Graph1} and for arbitrary graphs \cite{Graph2} concretely whereas we could use directly Blackadar's general construction. Only in the latter, we see Blackadar's paper being mentioned.
	
	As mentioned above, the algebras considered by Cuntz and Krieger were given in terms of a C*-algebra generated by operators in a Hilbert space. In \cite[Remark 2.15]{CuntzKrieger}, they mention that we can consider a relation given by an infinite sum, where the convergence is to be interpreted with respect to the strong operator topology. However, when looking at definitions of universal C*-algebras using relations, infinite sums are avoided as if it were not possible to find a universal representation. Perhaps, one reason for this avoidance is that one can think that the relations must be valid for an abstract C*-algebra and not in a represented form. This is the case, for instance, in \cite{CuntzLi,ExelLaca,Graph2, Graph1}. In \cite{InfiniteSum}, Boava and the author developed the theory of universal C*-algebras in order to allow for infinite sums relations with respect to the strong operator topology (see Section \ref{s:SOT.relation}). In \cite{InfiniteSum}, the examples of the Cuntz algebra $\CO_\infty$ and of the Exel-Laca algebras were explored. In these notes, we give the example of graph C*-algebras (Example \ref{ex:graph}).
	
	In \cite{Loring}, Loring extends Blackadar's construction. He defines a C*-relation in a categorical way and, under certain conditions, he shows the existence of a universal C*-algebra. However, it is not clear if the SOT relations defined in \cite{InfiniteSum} fit in the definition of C*-relation according to Loring. Informally, the main issue is that Loring asks for a relation to be preserved under *-homomorphism of C*-algebras \cite[Definition 2.2(C3)]{Loring}, and SOT relations depend on the representation, so there is some work to be done in order to show whether these are preserved by *-homomorphisms or not.
	
	These notes presume that the reader is familiar with basic notions on the theory of C*-algebras such as in \cite{LivroMurphy}. The structure is as follows: in Section \ref{s:free}, we give the description of some free objects, with the main goal of describing the free *-algebra generated by a set. In Section \ref{s:universal.algebra}, we give the construction of algebras and *-algebras given by generators and relations. In Section \ref{s:universal.C*}, the main goal is to give the construction of the universal C*-algebra given by generators and relations both with respect to the norm and with respect to the strong operator topology. In order to prove the existence of the universal C*-algebras, we need some conditions on the relations and the construction itself is done abstractly in Section \ref{s:admissible}.
	
	We fix the notation $\nn$ for non-negative integers and $\nn^*$ for positive integers.
	
	\section{Free *-algebras}\label{s:free}
	
	The main goal of this section is to build a free $*$-algebra generated by a set $X$. For that, we work first with other algebraic structures that are used as building blocks for the free $*$-algebra.
	
	\subsection{Free semigroups}\label{s:free.semigroup}
	
	Fix $X$ a non-empty set. For each $n\in\nn^*$, let $X^n=X\times\cdots\times X$ be the Cartesian product of $n$ copies of $X$. Denote an element of $X^n$ via concatenation, that is, for $x_1,\ldots,x_n\in X$, we write $x_1\cdots x_n$ instead of the usual $n$-tuple notation $(x_1,\cdots,x_n)$. Let $X^+=\bigcup_{n\in\nn^*}X^n$ and define an operation on $X^+$ by concatenation, that is,
	\[(x_1\cdots x_n)\cdot(y_1\cdots y_m)=x_1\cdots x_ny_1\cdots y_m\]
	where $x_1\cdots x_n\in X^n$ and $y_1\cdots y_m\in X^m$. This operation is clearly associative, so that $X^+$ with concatenation is a semigroup. The following theorem describes the universal property of $X^+$.
	
	\begin{theorem}\label{thm:ex.free.semigroup}
		Let $f:X\to S$ be a function from $X$ to a semigroup $S$. Then there exists a unique semigroup homomorphism $\tilde{f}:X^+\to S$ such that $\tilde{f}(x)=f(x)$ for all $x\in X$.
	\end{theorem}
	
	\begin{proof}
		Define $\tilde{f}:X^+\to S$ by $\tilde{f}(x_1\cdots x_n)=f(x_1)\cdots f(x_n)$, where $x_1\cdots x_n\in X^+$. By the definition of the product on $X^+$ and the associativity of the product in $S$, it is clear that $\tilde{f}:X^+\to S$ is a semigroup homomorphism such that $\tilde{f}(x)=f(x)$ for all $x\in X$.
		
		For the uniqueness, let $g: X^+\to S$ be a semigroup homomorphism such that $g(x)=f(x)$ for all $x\in X$. Then, for $x_1\cdots x_n\in X^+$, we have
		\[g(x_1\cdots x_n)=g(x_1)\cdots g(x_n)=f(x_1)\cdots f(x_n)=\tilde{f}(x_1\cdots x_n).\]
		Hence $g=\tilde{f}$.
	\end{proof}
	
	One can define a free semigroup generated by $X$ abstractly using the above property.
	
	\begin{definition}
		Let $X$ be a non-empty set. We say that $S_X$ is a \emph{free semigroup generated by $X$}, if there exists a function $i:X\to S_X$ such that for every function $f:X\to S$, where $S$ is a semigroup, there exists a unique semigroup homomorphism $\hat{f}:S_X\to S$ such that $f=\hat{f}\circ i$.
	\end{definition}
	
	We can interpret Theorem \ref{thm:ex.free.semigroup} as establishing existence of a free semigroup. In fact, any free semigroup is isomorphic to $X^+$. The proof is standard for algebraic structures satisfying a universal property. One way of having a generic proof is to interpret the construction in a categorical way. The uniqueness then is a consequence of the uniqueness, up to isomorphism, of an initial object in a category, in case it exists. Here, we prove it directly for semigroups. For the other constructions, the proofs are analogous and will be left as exercises to the reader.
	
	\begin{theorem}\label{thm:unique.free.semigroup}
		Let $X$ be a non-empty set and $S_X$ a free group generated by $X$. Then $X^+$ and $S_X$ are isomorphic semigroups.
	\end{theorem}
	
	\begin{proof}
		By Theorem, \ref{thm:ex.free.semigroup}, there exists a semigroup homomorphism $\tilde{i}:X^+\to S_X$ such that $\tilde{i}(x)=i(x)$ for all $x\in X$. On the other hand, if $h:X\to X^+$ is the inclusion function, then there exists a semigroup homomorphism $\hat{h}:S_X\to X^+$ such that $\hat{h}(i(x))=h(x)=x$ for all $x\in X$.
		
		Observe that the composition $\tilde{i}\circ\hat{h}:S_X\to S_X$ is such that $\tilde{i}(\hat{h}(i(x)))=\tilde{i}(x)=i(x)$ for all $x\in X$. Also, if $Id_{S_X}$ is the identity function on $S_X$, then $I_{S_X}\circ i=i$. The uniqueness part in the definition of a free semigroup then implies that $\tilde{i}\circ\hat{h}=\operatorname{Id}_{S_X}$. Similarly, we prove that $\hat{h}\circ\tilde{i}=\operatorname{Id}_{X^+}$. Hence, $\tilde{i}$ is an isomorphism between $X^+$ and $S_X$.
	\end{proof}
	
	\begin{remark}
		We can adapt the construction of the free semigroup to build a free monoid. In this case we let $X^0=\{\epsilon\}$, where $\epsilon$ is considered to be the empty word when concatenating. In this case, we let $X^*=\bigcup_{n\in\nn} X^n$ with operation given by concatenation. Similarly to what was done above, we prove that $X^*$ is the free monoid generated by $X$. Moreover, when working with monoids we can also consider the case where $X=\emptyset$.
	\end{remark}
	
	\subsection{Free modules}
	
	Before defining free algebras, let us work with free modules generated by a set. Fix $R$ be a ring with unit and $X$ a set.
	
	\begin{definition}
		We say that $M_X$ is a \emph{free $R$-module generated by $X$}, if there exists a function $i:X\to M_X$ such that for every function $\varphi:X\to M$, where $M$ is an $R$-module, there exists a unique $R$-module homomorphism $\hat{\varphi}:M_X\to M$ such that $\varphi=\hat{\varphi}\circ i$.
	\end{definition}
	
	If $X=\emptyset$, then the free module generated by $X$ is the zero module. So, from now on, let us suppose that $X$ is non-empty. Define
	\[R^{(X)}=\{f:X\to R:f(x)\neq 0\text{ for finitely many }x\in X\}.\]
	With pointwise operations $R^{(X)}$ is an $R$-module. For each $x\in X$, let $\delta_x:X\to R$ be given by $\delta_x(x)=1$ and $\delta_x(y)=0$ if $y\neq x$. Let $i:X\to R^{(X)}$ be given by $i(x)=\delta_x$.
	
	\begin{theorem}
		With the above construction, $R^{(X)}$ is a free $R$-module generated by $X$.
	\end{theorem}
	
	\begin{proof}
		Let $\varphi:X\to M$ be a function from $X$ to an $R$-module $M$. Define $\hat{\varphi}:R^{(X)}\to M$ by $\hat{\varphi}(f)=\sum_{x\in X}f(x)\varphi(x)$. Because $f(x)\neq 0$ for finitely many $x\in X$, the map $\hat{\varphi}(f)$ is well-defined. That $\hat{\varphi}$ is the unique homomorphism from $R^{(X)}$ to $M$ such that $\varphi=\hat{\varphi}\circ i$ is straightforward and left to the reader.
	\end{proof}
	
	\subsection{Free algebras}\label{s:free.algebra}
	
	Fix $R$ a commutative ring with unit and $X$ a set.
	
	\begin{definition}
		We say that $A_X$ is a \emph{free $R$-algebra generated by $X$} if there exists a function $i:X\to A_X$ such that for every function $\varphi:X\to A$, where $A$ is an $R$-algebra, there exists a unique $R$-algebra homomorphism $\hat{\varphi}:A_X\to A$ such that $\varphi=\hat{\varphi}\circ i$.
	\end{definition}
	
	As for modules, the zero algebra is the free $R$-algebra generated by the empty set. Suppose now that $X$ is non-empty and let $X^+$ be the free semigroup generated by $X$ as in Section \ref{s:free.semigroup}. As an $R$-module, define $R\left<X \right>:=R^{(X^+)}$. We define a multiplication on $R\left<X \right>$ by using the universal property of free modules and free semigroups.
	
	For each $x\in X$, define the function $L_x:X\to R\left<X \right>$ by $L_x(y)=xy$. By the universal property of $X^+$, there exists a unique semigroup homomorphism $\hat{L}_x:X^+\to R\left<X \right>$ extending $L_x$. Now by the universal property $R\left<X \right>$ as a free $R$-module, there exists a unique $R$-module homomorphism, $\tilde{L}_x:R\left<X \right>\to R\left<X \right>$ extending $\hat{L}_x$.
	
	Now we let $\psi:X\to\Hom_{R-mod}(R\left<X \right>,R\left<X \right>)$ be given by $\psi(x)=\tilde{L}_x$. Again, using the universal property of $X^+$ as a free semigroup, there exists a unique semigroup homomorphism $\hat{\psi}:X^+\to \Hom_{R-mod}(R\left<X \right>,R\left<X \right>)$ extending $\psi$. And by the universal property of $R\left<X \right>$ as a free $R$-module, there exists a unique $R$-module homomorphism $\tilde{\psi}:R\left<X \right>\to \Hom_{R-mod}(R\left<X \right>,R\left<X \right>)$ extending $\hat{\psi}$.
	
	Finally, we define the multiplication $m:R\left<X \right>\times R\left<X \right>\to R\left<X \right>$ by $m(p,q)=\tilde{\psi}(p)(q)$. By construction $m$ is bilinear. Using the bilinearity of $m$ and the associativity of concatenation, we prove that $m$ is associative.
	
	\begin{theorem}
		The $R$-algebra $R\left<X \right>$ is a free $R$-algebra.
	\end{theorem}
	
	\begin{proof}
		The function $i:X \to R\left<X \right>$ is the same as the one for free modules, observing that $X\scj X^+$. Given a function $\varphi:X\to A$, where $A$ is an $R$-algebra, we use the two steps as in the definition of the multiplication on $R\left<X \right>$. More precisely, there exists a semigroup homomorphism $\hat{\varphi}:X^+\to A$ extending $\varphi$, where we consider $A$ as a semigroup with the multiplication. Now, using the universal property of $R\left<X \right>$ as a free $R$-module, there exists a $R$-module homomorphism $\tilde{\varphi}:R\left<X \right>\to A$ extending $\hat{\varphi}$. That $\tilde{\varphi}$ preserves multiplication and is the unique $R$-algebra homomorphism such that $\tilde{\varphi}\circ i=\varphi$ is left to the reader.
	\end{proof}
	
	\begin{remark}\label{rmk:free.unital.algebra}
		If we want to work with $R$-algebra with unit, we can use a similar construction as above with the free monoid $X^*$ instead of the free semigroup $X^+$. The notation $R\left<X \right>$ is usually used for the $R$-algebra with unit, and it is called the $R$-algebra of polynomials with non-commuting variables.
	\end{remark}
	
	\subsection{Free *-algebras}
	
	Let $R$ be a commutative $*$-ring with unit and $X$ a set. We denote the involution on $R$ by $\overline{r}$ for $r\in R$. In what follows, $*$-algebra will mean $*$-algebra over $R$.
	
	\begin{definition}
		We say that $A_X$ is a \emph{free $*$-algebra generated by $X$}, if there exists a function $i:X\to A_X$ such that for every function $\varphi:X\to A$, where $A$ is a $*$-algebra, there exists a unique $*$-homomorphism $\hat{\varphi}:A_X\to A$ such that $\varphi=\hat{\varphi}\circ i$.
	\end{definition}
	
	As before, the zero algebra is the free $*$-algebra generated by the empty set. For the construction of the free $*$-algebra, we let $X^{\dagger}=\{x^*:x\in X\}$ be a disjoint copy of $X$. Consider $\falgx:=R\left<X\cup X^\dagger\right>$ the free algebra generated by $X\cup X^\dagger$. For simplicity of notation, we consider that $X\cup X^\dagger\scj \falgx$.
	
	Recall that to define a $*$-algebra, we need an involution that is a conjugate-linear anti-homomorphism. In order to define an involution on $\falgx$, we consider the $R$-algebra $\falgx^h$, such that $\falgx^h=\falgx$ as abelian groups with the addition. The other operations are given by $r\cdot_h p=\overline{r}p$ and $p\cdot_h q=qp$ for $r\in R$ and $p,q\in\falgx$. Consider the function $\varphi:X\cup X^\dagger$ given by $\varphi(x)=x^*$ and $\varphi(x^*)=x$. Using the universal property of $\falgx$ as an $R$-algebra, there is a unique $R$-algebra homomorphism $*:\falgx\to\falgx^h$. Going back to the original $R$-algebra structure of $\falgx$ on the codomain, we see that $*:\falgx\to \falgx$ is a conjugate-linear anti-homomorphism involution.
	
	\begin{theorem}
		The $*$-algebra $\falgx$ is a free $*$-algebra generated by $X$.
	\end{theorem}
	
	\begin{proof}
		Let $\varphi:X\to A$ be a function where $A$ is a $*$-algebra. Define a new function $\tilde{\varphi}:X\cup X^\dagger\to A$ by $\tilde{\varphi}(x)=\varphi(x)$ and $\tilde{\varphi}(x^*)=\varphi(x)^*$ for every $x\in X$. The universal property of $\falgx$ as a free $R$-algebra guarantees the existence of a $R$-algebra homomorphism $\hat{\varphi}$ extending $\tilde{\varphi}$. It is straight forward to check that $\hat{\varphi}$ is a $*$-homomorphism and the unique one that extends $\varphi$.
	\end{proof}
	
	\subsection{Free C*-algebras}
	
	Let $X$ be a set. We can mimic the definition of free object for C*-algebras.
	
	\begin{definition}
		We say that $C_X$ is a \emph{free C*-algebra generated by $X$}, if there exists a function $i:X\to C_X$ such that for every function $\varphi:X\to B$, where $B$ is a C*-algebra, there exists a unique $*$-homomorphism $\hat{\varphi}:C_X\to B$ such that $\varphi=\hat{\varphi}\circ i$.
	\end{definition}
	
	If $X$ is the empty set, then the zero algebra is the free C*-algebra generated by $X$. In all other case, we have the following.
	
	\begin{proposition}\label{prop:free.C*-algebra}
		There is no free C*-algebra generated by a non-empty set.
	\end{proposition}
	
	\begin{proof}
		Let $X$ be a non-empty set and suppose that $C_X$ is freely generated by $X$. Recall that every *-homomorphism between C*-algebras is contractive. For each $n\in\nn$, we consider the constant function $f_n:X\to \cn$ given by $f_n(x)=n$ for all $x\in X$. Then there exists a *-homomorphism $g_n:C_X\to\cn$ such that $g(x)=n$ and, hence $n\leq\|x\|$ for all $x\in X$ and $n\in\nn$. This is a contradiction with the fact that $\|x\|\in [0,\infty)$.
	\end{proof}
	
	\section{Universal algebras and *-algebras defined by generators and relations}\label{s:universal.algebra}
	
	It is often very useful to define an algebra generated by a certain set $X$ and then impose that two given expressions represent the same element. For instance, a way of defining the complex numbers is to consider the free $\rn$-algebra with unit generated by an element $i$ such that $i^2=-1$. In this section, we make the meaning of a relation more precise and show how to find an algebra satisfying a universal property. We later consider the case of $*$-algebras.
	
	\subsection{Algebras defined by generators and relations}
	
	Fix $X$ be a non-empty set and $R$ a commutative ring with unit. Consider also $R\left<X\right>$, the free algebra generated by $X$ defined in Section \ref{s:free.algebra}.
	
	\begin{definition}
		A \emph{relation} on the algebra $R\left<X\right>$, is a pair $(a,b)\in R\left<X\right>\times R\left<X\right>$. We say that a function $\varphi:X\to B$, where $B$ is an $R$-algebra, \emph{satisfies the relation} $(a,b)$ if $\hat{\varphi}(a)=\hat{\varphi}(b)$, where $\hat{\varphi}:R\left<X\right>\to B$ is the unique homomorphism given by the universal property of $R\left<X\right>$.
	\end{definition}
	
	A relation $(a,b)$ is to be interpreted as an equality $a=b$ to be satisfied by the variable in the set $X$, for instance, $i^2=-1$ for the complex numbers. Usually, we present a relation as an equality to be satisfied instead of a pair.
	
	\begin{definition}\label{def:univ.algebra}
		Given a set $\CR$ of relations on $R\left<X\right>$, a \emph{universal $R$-algebra generated by $X$ satisfying the relations in $\CR$} is an $R$-algebra $A$ with a function $\iota:X\to A$ that satisfies the relations in $\CR$ such that for every function $\varphi:X\to B$ satisfying the relations in $\CR$, there exists a unique homomorphism $\tilde{\varphi}:A\to B$ such that $\tilde{\varphi}\circ\iota=\varphi$.
	\end{definition}
	
	\begin{remark}
		It is also common to use the word free instead of universal in Definition \ref{def:univ.algebra}. In these notes, we use the word free only in the context presented in Section \ref{s:free}.
	\end{remark}
	
	\begin{theorem}\label{thm:ex.universa.algebra}
		For a set of relation $\CR$ on $R\left<X\right>$, there is a unique (up to isomorphism) $R$-algebra generated by $X$ satisfying the relations in $\CR$.
	\end{theorem}
	
	\begin{proof}
		For the existence, we consider $I$ to be the two-sided ideal of $R\left<X\right>$ generated by the set $\{a-b:(a,b)\in\CR\}$. Let $A=R\left<X\right>/I$ and define $\iota:X\to A$ by $\iota(x)=x+I$ (considering, by simplicity, that $X\scj R\left<X\right>$). The unique extension of $\iota$ to a homomorphism of $R$-algebras $\hat{\iota}:R\left<X\right>\to A$ is the quotient map, so that $\iota$ satisfies the relations in $\CR$.
		
		Consider a function $\varphi:X\to B$ satisfying the relation in $\CR$, where $B$ is an $R$-algebra, and let $\hat{\varphi}:R\left<X\right>\to B$ be the unique extension of $\varphi$ to $R\left<X\right>$. Note that $I\scj \ker(\hat{\varphi})$, so that it factors through an $R$-algebra homomorphism $\tilde{\varphi}:A\to B$ in the quotient. Clearly $\tilde{\varphi}\circ \iota=\varphi$.
		
		The proof of the uniqueness of $A$ follows the same lines as the proof of Theorem \ref{thm:unique.free.semigroup}.
	\end{proof}
	
	\begin{example}[Adding a unit]\label{ex:add.unit}
		Let $X$ be a set, and consider $u\notin X$. Let $A$ be the universal $R$-algebra generated by $X\cup\{u\}$ with relations given by
		\[u^2=u,\quad ux=x,\quad xu=u\quad\text{ for all }x\in X.\]
		Then $A$ is isomorphic to the unital $R$-algebra generated by $X$ as described in Remark \ref{rmk:free.unital.algebra}.
		
		The main difference is that for $A$ we do not require that $u$ is sent to the unit of another algebra, whereas if we are working in the category of unital algebras, the homomorphisms preserve the unit.
	\end{example}
	
	In what follows, when we say that universal algebra with unit, we mean that we are using the given generators and relations together with an extra element $u$ and the corresponding relations as in Example \ref{ex:add.unit}.
	
	\begin{example}[Complex numbers]
		The complex numbers is the universal $\rn$-algebra with unit generated by element $i$ satisfying the relation $i^2=-1$.
	\end{example}
	
	\begin{example}[Polynomial algebra with commuting variables]
		Let $X$ be a non-empty set. The polynomial algebra $R[X]$ is the universal $R$-algebra with unit satisfying the relations $xy=yx$ for all $x,y\in X$.
	\end{example}
	
	\begin{example}[Leavitt path algebras] \label{ex:Leavitt}
		A graph is a quadruple $E=(E^0,E^1,r,s)$, where $E^0$ is a set whose elements are called vertices, $E^1$ is a set whose elements are called edges, and $r,s:E^1\to E^0$ are functions called the range and source maps, respectively. We assume that $E^0$ and $E^1$ are disjoint. A vertex $v$ is said to be regular if $s^{-1}(v)$ is finite and non-empty. Also, let $(E^1)^*=\{e^*:e\in E^1\}$ be a copy of $E^1$ disjoint from $E^0\cup E^1$. The Leavitt path algebra $L_R(E)$ is the universal $R$-algebra generated by $E^0\cup E^1\cup (E^1)^*$ with relations:
		\begin{enumerate}
			\item $vv'=\delta_{v,v'}v$ for all $v,v'\in E^0$;
			\item $s(e)e=e=er(e)$ for all $e\in E^1$;
			\item $r(e)e^*=e^*=e^*s(e)$ for all $e\in E^1$;
			\item $f^*e=\delta_{e,f}r(e)$ for all $e\in E^1$;
			\item $v=\sum_{e\in s^{-1}(v)}ee^*$ for every regular vertex $v\in E^0$,
		\end{enumerate}
		where $\delta$ represents the Kronecker delta.
	\end{example}
	
	Note that in Example \ref{ex:Leavitt}, we considered a copy of $E^1$ similarly to what we have done in the construction of the free $*$-algebra. In fact, we can define a $*$-algebra generated by a set and relations.
	
	\subsection{*-algebras defined by generators and relations}
	
	From now on, suppose that $R$ is a $*$-ring. The definitions and results are analogous to those of the previous section.
	
	\begin{definition}
		A \emph{relation} on the algebra $\falgx$, is a pair $(a,b)\in \falgx\times \falgx$. We say that a function $\varphi:X\to B$, where $B$ is an $*$-algebra, \emph{satisfies the relation} $(a,b)$ if $\hat{\varphi}(a)=\hat{\varphi}(b)$, where $\hat{\varphi}:\falgx\to B$ is the unique $*$-homomorphism given by the universal property of $\falgx$.
	\end{definition}
	
	\begin{definition}\label{def:univ.*-algebra}
		Given a set $\CR$ of relations on $\falgx$, a \emph{universal $R$-algebra generated by $X$ satisfying the relations in $\CR$} is a $*$-algebra $A$ with a function $\iota:X\to A$  that satisfies the relations in $\CR$ such that for every function $\varphi:X\to B$ satisfying the relations in $\CR$, there exists a unique $*$-homomorphism $\tilde{\varphi}:A\to B$ such that $\tilde{\varphi}\circ\iota=\varphi$.
	\end{definition}
	
	\begin{theorem}
		For a set of relations $\CR$ on $\falgx$, there is a unique (up to isomorphism) $*$-algebra generated by $X$ satisfying the relations in $\CR$.
	\end{theorem}
	
	\begin{proof}
		The proof is analogous to that of Theorem \ref{thm:ex.universa.algebra}, replacing $I$ with the self-adjoint ideal generated by $\{a-b:(a,b)\in\CR\}$.
	\end{proof}
	
	\begin{example}[Leavitt path algebras]
		Let $E$ be a graph. We can define $L_R(E)$ as the $*$-algebra generated by $E^0\cup E^1$ satisfying the relations
		\begin{enumerate}
			\item $v^*=v$ for all $v\in E^0$;
			\item $vv'=\delta_{v,v'}v$ for all $v,v'\in E^0$;
			\item $s(e)e=e=er(e)$ for all $e\in E^1$;
			\item $f^*e=\delta_{e,f}r(e)$ for all $e\in E^1$;
			\item $v=\sum_{e\in s^{-1}(v)}ee^*$ for every regular vertex $v\in E^0$.
		\end{enumerate}
		If we start with a ring $R$ with no involution a priori, we can always consider the identity function as the involution for $R$. In this case, this definition coincides with the one given in Example \ref{ex:Leavitt}.
	\end{example}
	
	\section{Universal C*-algebras}\label{s:universal.C*}
	
	Even though we do not have free C*-algebras for non-empty sets, we can still build C*-algebras using generators and relations, as long as the relations guarantee that there is a certain bound in the norm of elements in the algebra being generated. The idea is to use the free $*$-algebra over $\cn$ as the starting point of our construction.
	
	From now on, all algebras will be over $\cn$.
	
	\subsection{Admissible class of representations}\label{s:admissible}
	
	Fix $\CG$ a non-empty set and $\falg$ the free $*$-algebra generated by $\CG$. A representation of $\falg$ on a Hilbert space is a $*$-homomorphism $\rho:\falg\to B(H)$, where $B(H)$ is the C*-algebra of bounded operators on $H$. Due to the universal property of $\CG$, it suffices to present a family of bounded operators $\{T_g:H\to H\}_{g\in\CG}$ in order to present $\rho$.
	
	\begin{definition}\label{admissible}
		A class $\CCC$ of representations of $\falg$ on Hilbert spaces (not necessarily the same) is said to be \emph{admissible} if for every family $\{\rho_\alpha:\falg\to B(H_\alpha)\}_\alpha$ of representations of $(\CG,\CR)$ in $\CCC$, the map
		\[\newfd{\bigoplus_\alpha\rho_{\alpha}}{\falg}{B\left(\bigoplus_\alpha H_\alpha\right)}{x}{\bigoplus\rho_\alpha(x)}\]
		is a well-defined representation of $\falg$ in the class $\CCC$.
	\end{definition}
	
	Given a non-empty class $\CCC$ of admissible representations of $\falg$, let $\CP_\CCC$ be the set of all seminorms $p:\falg\to\rn_+$ such that there exists a representation $\rho\in\CCC$ such that $p(a)=\|\rho(a)\|$ for all $a\in\falg$.
	
	\begin{remark}
		We use class in Definition~\ref{admissible}, because the collection of representations considered might be too large to be a set. In Proposition~\ref{C*-seminorm.well.defined}, to avoid taking the supremum over a class, we consider the subset $\CP_\CCC$ of the set of all functions from $\falg$ to $\rn_+$ and take the supremum over $\CP_\CCC$.
	\end{remark}
	
	\begin{proposition}\label{C*-seminorm.well.defined}
		In the context above, $\sup_{p\in\CP_\CCC}p(a)<\infty$ for all $a\in\falg$.
	\end{proposition}
	
	\begin{proof}
		Suppose that $\sup_{p\in\CP_\CCC}p(a)=\infty$ for some $a\in\falg$. Then, for each $n\in\nn$, there exists a representation $\rho_n:\falg\to B(H_n)$ such that $p(a)=\|\rho_n(a)\|\geq n$. In this case, $\bigoplus_n \rho_n$ is not a well-defined representation, which is a contradiction.
	\end{proof}
	
	Using the above proposition, we can define
	\begin{equation}\label{eq:C*-seminorm}\vvvert a\vvvert:=\sup_{p\in\CP_\CCC}\{p(a)\}\end{equation}
	for each $a\in\falg$. It is straightforward to check that $\vvvert\cdot\vvvert$ is a C*-seminorm in $\falg$, that is a seminorm such that $\vvvert a\vvvert^2=\vvvert a^*a\vvvert$ for all $a\in\falg$.
	
	We use the standard procedure to define a C*-algebra from a *-algebra with a C*-seminorm. Let $\CI:=\{a\in\falg:\vvvert a\vvvert=0\}$, which is an ideal of $\falg$. On the quotient $\falg/\CI$, we have the C*-norm given by
	\[\|a+\CI\|=\inf\{\vvvert a+b\vvvert:b\in\CI\}.\]
	We then define a C*-algebra $A(\CG,\CCC)$ as the completion of $\falg/\CI$. Then, $A(\CG,\CCC)$ is a C*-algebra with the natural operations. There is a map $\iota:\CG\to A(\CG,\CCC)$ by mapping $g\in\CG$ to its class in $\falg/I_\gp$.
	
	\subsection{Norm relations}\label{s:norm.relation}
	
	In this section, we explain the construction developed by Blackadar in \cite{Blackadar}.
	
	\begin{definition}
		We define a \emph{norm relation} as a pair $(x,\eta)\in\falg\times\rn_+$. A \emph{generating pair} is a pair $(\CG,\CR)$, where $\CR$ is a set of norm relations on $\CG$. We say that a $*$-homomorphism $\rho:\falg\to B(H)$, where $H$ is a Hilbert space, satisfies the relation $(x,\eta)$ if $\|\rho(x)\|\leq\eta$.
	\end{definition}
	
	A norm relation is to be interpreted as $\|x\|\leq \eta$. For an equality, we use $\eta=0$. And if we want to impose a relation of the form $a=b$ for $a,b\in\falg$, we consider the pair $(a-b,0)\in\falg\times\rn_+$.
	
	\begin{example}
		If we want to say that element $p\in\CG$ is a projection, we impose the relations $(p-p^*,0)$ and $(p-p^2,0)$. If we want to say that an element $s\in \CG$ is a partial isometry, we consider the relation $(ss^*s-s,0)$. For an element $1\in\CG$ to be a unit, we consider the relations $(1-1^*,0)$, $(1g-g,0)$ and $(g1-g,0)$ for every $g\in\CG$.
	\end{example}

	\begin{definition}
		A \emph{representation} of a generating pair $\gp$ is a *-homomorphism $\rho:\falg\to\CB(H)$, where $H$ is a Hilbert space, such that $\|\rho(x)\|\leq\eta$ for every $(x,\eta)\in\CR$. The generating pair $\gp$ is said to be \emph{admissible} if the class $\CCC$ of all representations of $\gp$ is admissible. In this case, we define the universal C*-algebra generated by $\CG$ with relations $\CR$ as $C^*\gp=A(\CG,\CCC)$.
	\end{definition}
	
	\begin{proposition}\label{prop:admissible.norm}
		A generating pair $(\CG,\CR)$ is admissible if and only if for every $g\in\CG$, there is a constant $\lambda_g$ such that $\|\rho(g)\|\leq\lambda_g$ for all representation $\rho$ of $(\CG,\CR)$.
	\end{proposition}
	
	\begin{proof}
		If $\gp$ is admissible, we can use Proposition \ref{C*-seminorm.well.defined} and equation \eqref{eq:C*-seminorm} to set $\lambda_g=\vvvert g\vvvert$.
		
		Suppose that there exist the constants $\lambda_g$. Using properties of norm, it is easy to check that for every $y\in\falg$, there exists $\lambda_y\in\rn_+$ such that $\|\rho(y)\|\leq\lambda_y$ for all representation $\rho$ of $(\CG,\CR)$. In particular, for any family of representation $\{\rho_\alpha\}_\alpha$ of $(\CG,\CR)$, the direct sum $\rho_\alpha$ is a well-defined *-homomorphism. And for every $(x,\eta)\in\CR$, we have that
		\[\left\|\bigoplus_{\alpha}\rho_{\alpha}(x)\right\|=\sup_{\alpha}\|\rho_\alpha(x)\|\leq\eta.\]
		Hence the class of all representations of $\gp$ is admissible.
	\end{proof}
	
	\begin{proposition}
		For every pair $(x,\eta)\in\CR$, the inequality $\|x\|\leq \eta$ holds in  $C^*\gp$.
	\end{proposition}
	
	\begin{proof}
		For each $z\in\CI=\{a\in\falg:\vvvert a\vvvert =0\}$, we have $\vvvert x+z\vvvert \leq\vvvert x\vvvert +\vvvert z\vvvert =\vvvert x\vvvert\leq\eta$, since for every representation $\rho$ of $\gp$, $\|\rho(x)\|\leq \eta$. Hence, $\|x\|\leq \eta$ in $C^*\gp$.
	\end{proof}
	
	\begin{theorem}
		Let $\gp$ be a generating pair, $A$ be a C*-algebra and $f:\CG\to A$ a function. If the unique *-homomorphism $\hat{f}:\falg\to A$ is such that $\|\hat{f}(x)\|\leq\eta$ for every $(x,\eta)\in\CR$, then there exists a unique *-homomorphism $\tilde{f}:C^*\gp\to A$ such that $\tilde{f}(g)=f(g)$ for every $g\in\CG$.
	\end{theorem} 
	
	\begin{proof}
		Because every C*-algebra can be thought of as being a concrete C*-algebra (that is, a sub-C*-algebra of $B(H)$ for some Hilbert space $H$), $\hat{f}$ is a representation of $\gp$, so that for every $y\in\falg$, we have $\|\hat{f}(y)\|\leq \vvvert y\vvvert$. In particular $\CI=\{a\in\falg:\vvvert a\vvvert =0\}\scj\ker(\hat{f})$. The map can be factored through the quotient, and it is contractive. Hence, it can be extended to its completion. The uniqueness follows from the fact that any element of $C^*\gp$ can be approximated by polynomials on the generators and its involutions.
	\end{proof}

	\begin{example}
		Consider the pair $\gp$ given by $\CG=\{1,x\}$ and \[\CR=\{(1^2-1,0),(1^*-1,0),(1x-x,0),(x1-x,0),(x^+x-xx^+,0),(x,1)\}.\] The pair is admissible by Proposition \ref{prop:admissible.norm} by taking $\lambda_1=1$ and $\lambda_x=1$. We say that $C^*\gp$ is the unital universal C*-algebra generated by $x$ with relations $x^+x=xx^+$ and $\|x\|\leq 1$.
		
		Let $\overline{\mathbb{D}}$ be the closed disc and $z$ the identity function on it. By the universal property, there exists a *-homomorphism $\varphi:C^*\gp\to C(\overline{\mathbb{D}})$ such that $\varphi(x)=z$. It is surjective by the Stone-Weierstrass theorem. Because $\|x\|\leq 1$, we have that $\sigma(x)\scj\overline{\mathbb{D}}$. Hence, for any  polynomial $f\in C[y,y^*]$ on variables $y$ and $y^*$, we have that $\|f(x)\|\leq \|f(z)\|$. The inequality $\|f(z)\|\leq\|f(x)\|$ follows from the fact that $\varphi$ is contractive. That means that $\varphi$ is an isometry on a dense subset of $C^*\gp$ and hence injective.
	\end{example}
	
	\begin{example}
		Consider $A$ the universal C*-algebra generated by an isometry $s$. The corresponding generating pair is admissible since every isometry has norm bounded by 1. Let $S$ be the shift operator on $\ell^2(\nn)$. Then there is a *-homomorphism $\varphi:A\to B(\ell^2(\nn))$ such that $\varphi(s)=S$. In particular, $s$ is not unitary. Taking any faithful representation of $A$ and applying Coburn's theorem \cite{Coburn}, we have that $A\cong C^*(S)\cong\CT$.
	\end{example}
	
	\begin{example}
		The pair $(\{x\},\{(x-x^*,0)\})$ is not admissible. This follows from the fact that each real number is self-adjoint in $\cn$ and there is no bound on the norm of all real numbers.
	\end{example}
	
	\subsection{SOT relations}\label{s:SOT.relation} In this section, we consider relations described via convergence with respect to the strong operator topology, as done in \cite{InfiniteSum}. One of the author's and Boava's main motivation for this construction is the possibility to generalize the Cuntz-Li construction of a C*-algebra of an integral domain without the assumption that each quotient of the domain by a principal ideal is finite done in \cite[Section 2]{CuntzLi}. This generalization will be explored in a future work.
	
	\begin{definition}
		Let $\CG$ be a set and $\falg$ the free *-algebra generated by $\CG$. A \emph{SOT relation} is a net $(x_i)_{i\in I}$. A representation $\rho:\falg\to B(H)$ \emph{satisfies the relation $(x_i)_{i\in I}$} if $s-\lim \rho(x_i)=0$, where $s-\lim$ represents the limit in the strong operator topology. A \emph{generating triple} is a triple $\triple$, where $\CG$ is a set, $\CR_N$ is a set of norm relations and $\CR_S$ is a set of SOT relations.
	\end{definition}
	
	We now consider norm relations together with SOT relations.
	
	\begin{definition}
		A representation of a generating triple $\triple$ is a *-homomorphism $\rho:\falg\to\CB(H)$ satisfying the relations in $\CR_N$ and $\CR_S$, where $H$ is a Hilbert space. The generating triple $\triple$ is said to be admissible if the class $\CCC$ of all representations of $\triple$ is admissible. In this case, we define the universal C*-algebra generated by $\CG$ with relations $\CR_N$ and $\CR_S$ as $C^*\triple=A(\CG,\CCC)$.
	\end{definition}
	
	We prove a sufficient condition for the admissibility of a generating triple.
	
	\begin{proposition}
		Let $\triple$ be a generating triple. If there are constants $\lambda_g$ for each $g\in G$ and $\lambda_{\bx}$ for each $\bx=(x_i)_{i\in I}\in\CR_S$ such that for every representation $\rho$ of $\triple$, we have $\|\rho(g)\|\leq\lambda_g$ for all $g\in G$, and $\|\rho(x_i)\|\leq\eta_{\bx}$ for all $\bx=(x_i)_{i\in I}\in\CR_S$ and all $i\in I$, then the $\triple$ is admissible.
	\end{proposition}
	
	\begin{proof}
		Let $\{\rho_\alpha\}_{\alpha}$ be a family of representations of $\triple$. As in Proposition \ref{prop:admissible.norm}, the direct sum $\rho:=\bigoplus_\alpha\rho_\alpha$ is a representation of $\falg$ satisfying the norm relations. It remains to prove that $\rho$ also preserves the SOT relations.
		
		For each $\xi,\xi'\in\bigoplus_\alpha H_\alpha$, we have that $\left\|\bigoplus_\alpha \rho_\alpha(x_i)(\xi-\xi')\right\|\leq \eta_\bx\|\xi-\xi'\|$ for all $i\in I$. Moreover, if $\xi'$ is supported in a finite amount of indices, then there for each $\varepsilon >0$, there exists $i_0\in I$ such that $\left\|\bigoplus_\alpha \rho_\alpha(x_i)(\xi')\right\|\leq \varepsilon$ for all $i\geq i_0$. Since $\xi$ can be arbitrarily approximated by an element with finite support, we see that $\left\|\bigoplus_\alpha \rho_\alpha(x_i)(\xi)\right\|$ can be made arbitrarily small for sufficiently large $i\in I$.
	\end{proof}

	\begin{theorem}\label{thm:universal.property}
		Let $\triple$ be an admissible generating triple and let $A$ be a C*-algebra. If $\varphi:\falg\to A$ is a *-homomorphism such that there exists a faithful representation $\pi:A\to B(H)$ of $A$ on a Hilbert space $H$ for which $\pi\circ\varphi$ is a representation of $\triple$, then there is a unique *-homomorphism $\widetilde{\varphi}:C^*\triple\to A$ such that $\varphi=\widetilde{\varphi}\circ i$.
	\end{theorem}
	
	\begin{proof}
		Note that $\pi\circ\varphi$ is contractive and vanishes on $\CI$ by the definition of $\vvvert\cdot\vvvert$. Using that $\pi:A\to B(H)$ is faithful, we can define $\widetilde{\varphi}:C^*\triple\to A$ asking that $\widetilde{\varphi}(x+\CI)=\varphi(x)$ for $x\in\falg$ and then passing to the completion.
	\end{proof}
	
	We now show that there exists a faithful representation of $C^*\triple$ such that the composition with the natural map $\iota:\falg\to C^*\triple$ is a representation of $\triple$ satisfying the relations in $\CR_N$ and $\CR_S$. It is worth pointing out that not all faithful representations of $C^*\triple$ have this property (see \cite[Section~3.1]{InfiniteSum}).
	
	\begin{theorem}\label{thm:faithful.rep}
		Let $\triple$ be an admissible generating triple. Then, there exists a faithful representation $\pi_u:C^*\triple\to B(H_u)$ on a Hilbert space $H_u$ such that $\pi_u\circ \iota$ is a representation of $\triple$.
	\end{theorem}
	
	\begin{proof}[Proof idea]
		The idea of the proof is first to use the definition of the seminorm $\vvvert\cdot\vvvert$ in $\falg$, to find a representation $\rho_x$ of the triple $\triple$ for each $x\in\falg\setminus\CI$ such that $\vvvert x\vvvert=\|\rho_x\|$. Since the triple $\triple$ is admissible, we can consider the direct sum representation of the triple $\rho_u=\bigoplus_{x\in\falg\setminus\CI}\rho_x:\falg\to B(H_u)$, where $H_u$ is taken to be the direct sum of the Hilbert spaces of each representation $\rho_x$. By Theorem \ref{thm:universal.property}, there exists a representation $\pi_u:C^*\triple\to B(H_u)$ such that $\pi_u\circ\iota =\rho_u$,. Then, we use that any non-zero element of $C^*\triple$ can be sufficiently approximated by $\iota(x)$ for some $x\in\falg\setminus\CI$ to prove that $\pi_u(a)\neq 0$. For the details of the proof, see \cite[Theorem~2.15]{InfiniteSum}
	\end{proof}
	
	\begin{proposition}
		If $B$ is a C*-algebra satisfying the universal property of $C^*\triple$, then $B\cong C^*\triple$.
	\end{proposition}
	
	\begin{proof}
		The proof follows similar arguments to those in the proof of Theorem \ref{thm:unique.free.semigroup} and is left to the reader.
	\end{proof}
	
	\begin{example}\label{ex:graph}
		Let $E=(E^0,E^1,r,s)$ be a graph, where $E^0$ is the set of vertices, $E^1$ is the set of edges, $s:E^1\to E^0$ is the source map and $r:E^1\to E^0$ is the range map. A vertex $v$ is called a sink if $|s^{-1}(v)|=0$, a regular vertex if $0<|s^{-1}(v)|<\infty$ and an infinite emitter if $|s^{-1}(v)|=\infty$. The sets of sinks, regular vertices and infinite emitters are denoted respectively by $E^0_{sink}$, $E^0_{rg}$ and $E^0_{inf}$.
		Following \cite{Graph2}, $C^*(E)$ is the universal C*-algebra generated by mutually orthogonal projections $\{p_v\}_{v\in E^0}$ and partial isometries with mutually orthogonal final projections $\{s_e\}_{e\in E^1}$ satisfying:
		\begin{itemize}
			\item (CK1) $s_e^*s_e=p_{r(e)}$ for all $e\in E^1$,
			\item (CK2) $s_es_e^*\leq p_{s(e)}$ for all $e\in E^1$,
			\item (CK3) $p_v=\sum_{e\in s^{-1}(v)}s_es_e^*$ for all $v\in E^0_{rg}$.
		\end{itemize}
		
		Let $A$ be the universal C*-algebra generated by mutually orthogonal projections $\{q_v\}_{v\in E^0}$ and partial isometries with mutually orthogonal final projections $\{t_e\}_{e\in E^1}$ satisfying:
		\begin{itemize}
			\item $t_e^*t_e=q_{r(e)}$ for all $e\in E^1$,
			\item $q_v=\sum_{e\in s^{-1}(v)}t_et_e^*$ for all $v\in E^0\setminus E^0_{sink}$.
		\end{itemize}
		For $v\in E^0_{inf}$ the last relation is to be interpreted as the net $(p_v-\sum_{e\in F}s_es_e^*)_{F\in\CP_{fin}(s^{-1}(v))}$. First, we notice that the triple describing $A$ is admissible. Indeed, projections and partial isometries have norm less or equal than 1. And for every $v\in E^0_{inf}$ and $F\scj s^{-1}(v)$, we have that $p_v-\sum_{e\in F}s_es_e^*$ has norm at most 2.
		
		We build a homomorphism from $C^*(E)$ to $A$ that sends $p_v$ to $q_v$ and $s_e$ to $t_e$ for each $v\in E^0$ and $e\in E^1$. Most relations defining $C^*(E)$ are included in the relations defining $A$. The only thing that remains to prove is that $t_et_e^*\leq q_{s(e)}$ for all $e\in E^1$. Let $\pi_u$ be a representation as in Theorem \ref{thm:faithful.rep}. Call $T_e=\pi_u(t_e)$ and $Q_{s(e)}=\pi_u(q_{s(e)})$. Because, $\pi_u$ is faithful, it is sufficient to prove that $T_eT_e^*\leq Q_{s(e)}$, that is, $T_eT_e^*=Q_{s(e)}T_eT_e^*$. The limit below will be taken over the directed set $\CP_{fin}(s^{-1}(v))$. For any $\xi\in H_u$, we have
		\[Q_{s(e)}T_eT_e^*\xi=\lim_F\sum_{f\in F}T_fT_f^*T_eT_e^*\xi=\lim_FT_eT_e^*\xi=T_eT_e^*\xi.\]
		Since $\xi$ was arbitrary, the equality $T_eT_e^*\leq Q_{s(e)}$ holds.
		
		By the universal property of $C^*(E)$ there exists a homomorphism $\Phi:C^*(E)\to A$ such that $\Phi(p_v)=q_v$ and $\Phi(s_e)=t_e$. Since all the generators of $A$ are in the image, $\Phi$ is surjective.
		
		In order to prove that $\Phi$ is injective, we use the gauge-invariance theorem. For that, we need a strongly continuous action of $\T$ on $A$. Given $z\in\T$, we consider the *-homomorphism $\gamma_z:\falg\to A$ given by $\gamma_z(p_v)=p_v$ and $\gamma_z(s_e)=zs_e$. It is straightforward to check all norm relations. As for the SOT relations, we use a faithful representation $\pi_u$ as before. It is then easy to see that $\sum_{e\in s^{-1}(v)}(zT_e)(zT_e)^*$ converges to $P_v$ in the strong operator topology. Clearly $\gamma_{\overline{z}}=\gamma_z^{-1}$, so that we have an action of $\T$ on $A$. A typical $\varepsilon/3$ shows that the action is strongly continuous.
		
		Moreover, if $\sigma$ is the gauge action on $C^*(E)$, it is also immediate that $\Phi\circ\sigma_z=\gamma_z\circ\Phi$ for all $z\in\T$. By the gauge-invariance theorem \cite[Theorem~2.1]{Graph3}, $\Phi$ is injective, and hence a *-isomorphism.
	\end{example}


\begin{thebibliography}{10}
		
		\bibitem{UniversalCuntzKrieger}
		A.~an~Huef and I.~Raeburn.
		\newblock The ideal structure of {Cuntz}-{Krieger} algebras.
		\newblock {\em Ergodic Theory Dyn. Syst.}, 17(3):611--624, 1997.
		
		\bibitem{Graph3}
		T.~Bates, J.~H. Hong, I.~Raeburn, and W.~Szyma{\'n}ski.
		\newblock The ideal structure of the {{\(C^*\)}}-algebras of infinite graphs.
		\newblock {\em Ill. J. Math.}, 46(4):1159--1176, 2002.
		
		\bibitem{Blackadar}
		B.~Blackadar.
		\newblock Shape theory for {$C^\ast$}-algebras.
		\newblock {\em Math. Scand.}, 56(2):249--275, 1985.
		
		\bibitem{InfiniteSum}
		G.~Boava and G.~G. de~Castro.
		\newblock Infinite sum relations on universal
		{{\(\mathrm{C}^{\ast}\)}}-algebras.
		\newblock {\em J. Math. Anal. Appl.}, 507(2):16, 2022.
		\newblock Id/No 125850.
		
		\bibitem{Coburn}
		L.~A. Coburn.
		\newblock The {$C^{\ast} $}-algebra generated by an isometry.
		\newblock {\em Bull. Amer. Math. Soc.}, 73:722--726, 1967.
		
		\bibitem{Cuntz}
		J.~Cuntz.
		\newblock Simple {$C\sp*$}-algebras generated by isometries.
		\newblock {\em Comm. Math. Phys.}, 57(2):173--185, 1977.
		
		\bibitem{CuntzKrieger}
		J.~Cuntz and W.~Krieger.
		\newblock A class of {$C^{\ast} $}-algebras and topological {M}arkov chains.
		\newblock {\em Invent. Math.}, 56(3):251--268, 1980.
		
		\bibitem{CuntzLi}
		J.~Cuntz and X.~Li.
		\newblock The regular {$C^\ast$}-algebra of an integral domain.
		\newblock In {\em Quanta of maths}, volume~11 of {\em Clay Math. Proc.}, pages
		149--170. Amer. Math. Soc., Providence, RI, 2010.
		
		\bibitem{ExelLaca}
		R.~Exel and M.~Laca.
		\newblock Cuntz-{K}rieger algebras for infinite matrices.
		\newblock {\em J. Reine Angew. Math.}, 512:119--172, 1999.
		
		\bibitem{Graph2}
		N.~J. Fowler, M.~Laca, and I.~Raeburn.
		\newblock The {{\(C^*\)}}-algebras of infinite graphs.
		\newblock {\em Proc. Am. Math. Soc.}, 128(8):2319--2327, 2000.
		
		\bibitem{Graph1}
		A.~Kumjian, D.~Pask, and I.~Raeburn.
		\newblock Cuntz-{K}rieger algebras of directed graphs.
		\newblock {\em Pacific J. Math.}, 184(1):161--174, 1998.
		
		\bibitem{Loring}
		T.~A. Loring.
		\newblock {{\(C^\ast\)}}-algebra relations.
		\newblock {\em Math. Scand.}, 107(1):43--72, 2010.
		
		\bibitem{LivroMurphy}
		G.~J. {Murphy}.
		\newblock {\em {C\({}^*\)-algebras and operator theory}}.
		\newblock Boston, MA etc.: Academic Press, Inc., 1990.
		
	\end{thebibliography}
\end{document}